\documentclass[a4paper]{llncs}

\usepackage[utf8x]{inputenc}

\usepackage{amsmath}

\usepackage{amsthm}
\usepackage{amssymb}
\usepackage{breqn}
\usepackage{framed}
\usepackage{array}
\usepackage{xfrac}
\usepackage{tikz}
\usepackage{tkz-berge}
\usepackage{enumerate}
\usepackage{float}
\usepackage{bookmark}
\usepackage{hyperref}
\usepackage{longtable}
\usepackage{graphicx}

\hypersetup{
  pdftitle = {Integer Complexity: Experimental and Analytical Results II},
  pdfauthor = {Juris Čerņenoks, Jānis Iraids, Mārtiņš Opmanis, Rihards Opmanis, Kārlis Podnieks},
  pdfkeywords = {integer complexity, logarithmic complexity, spectrum, powers of two, ternary representations, randomness of pi},
  pdfstartview=FitH,
  unicode=true
}

\begin{document}

\bookmarksetup{startatroot}

\title{Integer Complexity: Experimental and Analytical Results II}
\author{\texorpdfstring{Juris Čerņenoks\inst{1} \and Jānis Iraids\inst{1} \and Mārtiņš Opmanis\inst{2} \and Rihards Opmanis\inst{2} \and Kārlis Podnieks\inst{1}}{Juris Čerņenoks, Jānis Iraids, Mārtiņš Opmanis, Rihards Opmanis, Kārlis Podnieks}}

\institute{University of Latvia, Raiņa bulvāris 19, Riga, LV-1586, Latvia \and
Institute of Mathematics and Computer Science, University of Latvia, Raiņa bulvāris 29, Riga, LV-1459, Latvia}

\theoremstyle{plain}
\newtheorem{hypo}{Hypothesis}
\newtheorem{cor}{Corollary}
\newtheorem{que}{Question}
\newtheorem{obs}{Observation}

\newcommand{\f}[1]{ \left\| #1 \right\| }
\newcommand{\fl}[1]{ \f{ #1 }_{\log} }
\newcommand{\fm}[1]{ \f{ #1 }_{-} }
\newcommand{\fml}[1]{ \f{ #1 }_{-\log} }
\newcommand{\fp}[1]{ \f{ #1 }_{P} }
\newcommand{\fpl}[1]{ \f{ #1 }_{P,\log} }

\maketitle

\bookmarksetup{startatroot}

\begin{abstract}
We consider representing of natural numbers by expressions using 1's, addition, multiplication and parentheses. $\f{n}$ denotes the minimum number of 1's in the expressions representing $n$. The logarithmic complexity $\fl{n}$ is defined as ${\f{n}}/{\log_3 n}$. The values of $\fl{n}$ are located in the segment $[3, 4.755]$, but almost nothing is known with certainty about the structure of this ``spectrum'' (are the values dense somewhere in the segment etc.). We establish a connection between this problem and another difficult problem: the seemingly ``almost random'' behaviour of digits in the base 3 representations of the numbers $2^n$.

We consider also representing of natural numbers by expressions that include subtraction, and the so-called $P$-algorithms - a family of ``deterministic'' algorithms for building representations of numbers.

\end{abstract}

\pdfbookmark[1]{1. Introduction}{intro}
\section{Introduction}
\label{intro}

The field explored in this paper is represented in ``The On-Line Encyclopedia of Integer Sequences'' as the sequences A005245 \cite{oeisA005245} and A091333 \cite{oeisA091333}. The topic seems gaining  popularity - see \cite{altman}, \cite{altman2}, \cite{altman3}, \cite{steinerberger}, \cite{arias2}.

The paper continues our previous work \cite{paper1}.

First, in Section \ref{int} we consider representing of natural numbers by arithmetical expressions using 1's, addition, multiplication and parentheses. Let's call this ``representing numbers in basis $\{1,+,\cdot\}$''. 

\begin{definition}
Let's denote by $\f{n}$ the \textbf{minimum} number of 1's in the expressions representing $n$ in basis $\{1,+,\cdot\}$. We will call it the \textbf{integer complexity} of $n$. The logarithmic complexity $\fl{n}$ is defined as $\frac{\f{n}}{\log_3 n}$.
\end{definition}
It is well known that all the values of $\fl{n}$ are located in the segment $[3, 4.755]$, but almost nothing is known with certainty about the structure of this ``spectrum'' (are the values dense somewhere in the segment etc.). We establish a connection between this problem and another difficult problem: the seemingly ``almost random'' behaviour of digits in the base 3 representations of the numbers $2^n$. 

Secondly, in Section \ref{minus} we consider representing of natural numbers by arithmetical expressions that include also \textbf{subtraction}. Let's call this ``representing numbers in basis $\{1,+,\cdot, -\}$''. 

\begin{definition}
\label{compl_minus}
Let's denote by $\fm{n}$ the \textbf{minimum} number of 1's in the expressions representing $n$ in basis $\{1, +, \cdot, -\}$. The logarithmic complexity $\fml{n}$ is defined as $\frac{\fm{n}}{\log_3 n}$.
\end{definition}

We prove that almost all values of the logarithmic complexity $\fml{n}$ are located in the segment $[3, 3.679]$. Having computed $\fm{n}$ up to $n=2 \cdot 10^{11}$, we present some of our observations.  

In Section \ref{palgo} we explore the so-called $P$-algorithms - a family of ``deterministic'' algorithms for building representations of numbers in basis $\{1,+,\cdot\}$. ``Deterministic'' means that these algorithms do not use searching over trees, but are building expressions directly from the numbers to be represented.

Let $P$ be a non-empty finite set of primes, for example, $P=\{2\}$, or $P=\{5, 11\}$. $P$-algorithm is building an expression of a number $n>0$ in basis $\{1,+,\cdot\}$ by subtracting 1's and by dividing (whenever possible) by primes from the set $P$. We explore the spectrum of the logarithmic complexity $\fpl{n} = \frac{\fp{n}}{\log_3 n}$.

\pdfbookmark[1]{2. Integer complexity in basis $1, +, .$}{int}
\section{Integer complexity in basis $\{1,+,\cdot\}$}
\label{int}

\pdfbookmark[2]{2.1. Connections to sum-of-digits problem}{conn}
\subsection{Connections to sum-of-digits problem}
\label{subsec:conn}

Throughout this subsection, we assume that $p, q$ are positive integers such that  $\frac{\log{p}} {\log{q}}$ is irrational, i.e., $p^a \neq q^b$ for any integers $a, b>0$.

\begin{definition}
Let us denote by $D_q(n, i)$ the \textbf{$i$-th digit} in the canonical base $q$ representation of the number $n$, and by $S_q(n)$ - the \textbf{sum of digits} in this representation. 
\end{definition}

Let us consider base $q$ representations of powers $p^n$. Imagine, for a moment (somewhat incorrectly), that, for fixed $p, q, n$,  the digits  $D_q(p^n, i)$ behave like as statistically independent random variables taking the values $0, 1, ..., q-1$ with equal probabilities $\frac{1}{q}$. Then, the (pseudo) mean value and (peudo) variance of $D_q(p^n, i)$ would be
 \[E=\frac{q-1}{2}; V=\sum\limits_{i=0}^{q-1}\frac{1}{q}\left(i-\frac{q-1}{2}\right)^2=\frac{q^2-1}{12}.\]
The total number of digits in the base $q$ representation of $p^n$ is $k_n \approx n \log_q{p}$, hence, the (pseudo) mean value of the sum $S_q(p^n)=\sum \limits_{i=1}^{k_n} D_q(p^n, i)$ would be $E_n \approx n\frac{q-1}{2}\log_q{p}$ and, because of the assumed (pseudo) independence of digits, its (pseudo) variance would be $V_n \approx n\frac{q^2-1}{12}\log_q{p}$. As the final consequence, the corresponding centered and normed variable $\frac{S_q(p^n)-E_n}{\sqrt{V_n}}$ would behave as a standard normally distributed random variable with probability density $\frac{1}{\sqrt{2\pi}}e^{-\frac{x^2}{2}}$.

One can try verifying this conclusion experimentally. For example, let us compute $S_3(2^n)$ for $n$ up to $100000$, and let us draw the histogram of the corresponding centered and normed variable
\[s_3(2^n)=\frac{S_3(2^n)-n \log_3 2}{\sqrt{n\frac{2}{3}\log_3 2 }}\]
(see Fig. \ref{Histo1}). As we see, this variable behaves, indeed, almost exactly, as a standard normally distributed random variable (the solid curve).

\begin{figure}
    \centering
    \includegraphics[width=0.8\textwidth]{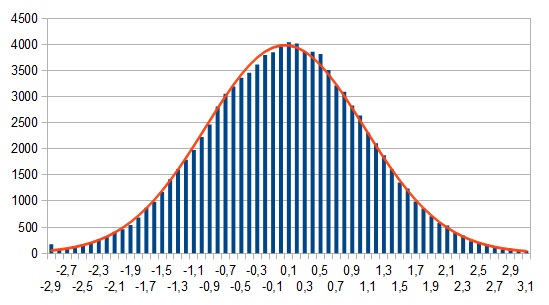}
    \caption{Histogram of centered and normed variable $s_3(2^n)$}
    \label{Histo1}
\end{figure}

Observing such a phenomenon ``out there'', one could conjecture that $S_q(p^n)$, as a function of $n$, behaves almost as $n\frac{q-1}{2}\log_q{p}$, i.e., almost \textbf{linearly} in $n$. Let us try to estimate the amplitude of the possible deviations by ``applying'' the Law of the Iterated Logarithm. Let us introduce centered and normed (pseudo) random variables:
\[d_q(p^n, i) = \frac{D_q(p^n, i)-\frac{q-1}{2}}{\sqrt{\frac{q^2-1}{12}}}.\]
By summing up these variables for $i$ from $1$ to $k_n$, we obtain a sequence of (pseudo) random variables:
\[\kappa_q(p, n) = \frac{S_q(p^n)-\frac{q-1}{2}k_n}{\sqrt{\frac{q^2-1}{12}}},\]
that ``must obey" the Law of the Iterated Logarithm. Namely, if the sequence $S_q(p^n)$ behaves, indeed, as a "typical" sum of equally distributed random variables, then $\lim \limits_{n \to \infty} \inf$ and $\lim \limits_{n \to \infty} \sup$ of the fraction
\[\frac{\kappa_q(p, n)}{\sqrt{2 k_n \log{\log k_n}}},\]
 ($\log$ stands for the natural logarithm) must be $-1$ and $+1$ correspondingly.

Therefore, it seems, we could conjecture that, if we denote
\[\sigma_q(p, n) = \frac{S_q(p^n)-(\frac{q-1}{2}\log_q p)n}{\sqrt{(\frac{q^2-1}{6}\log_q p) n \log{\log n}}},\]
then
\[\lim \limits_{n \to \infty} \sup \sigma_q(p, n)=1; \lim \limits_{n \to \infty} \inf \sigma_q(p, n) = -1\].

In particular, this would mean that
\[S_q(p^n)=(\frac{q-1}{2}\log_q p)n +O(\sqrt{n \log \log n}).\]
By setting $p=2; q=3$ (note that $\log_3 2 \approx 0.6309$):
\[S_3(2^n)=n \cdot \log_3 2 + O(\sqrt{n \log \log n});\]
\[\sigma_3(2, n) =\frac{S_3(2^n)-n \log_3 2}{\sqrt{(\frac{4}{3} \log_3 2) n \log{\log n}}}\approx \frac{S_3(2^n)-0.6309 n}{\sqrt{0.8412 n  \log \log n}},\] 
\[\lim \limits_{n \to \infty} \sup \sigma_3(2, n)=1; \lim \limits_{n \to \infty} \inf \sigma_3(2, n) = -1\].
 
\begin{figure}
    \centering
    \includegraphics[width=0.8\textwidth]{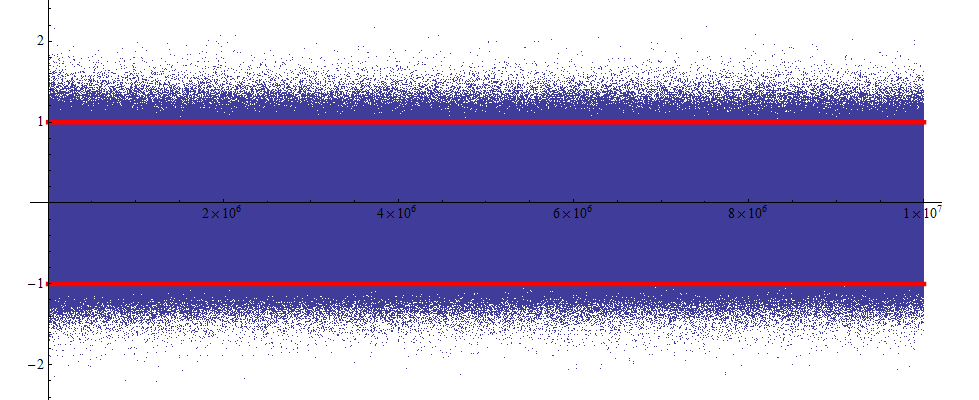}
    \caption{Oscillating behaviour of the expression $\sigma_3(2, n)$}
    \label{Oscilat1}
\end{figure}

However, the behaviour of the expression $\sigma_3(2, n)$ until $n=10^7$ does not show convergence to the segment $[-1, +1]$ (see Fig. \ref{Oscilat1}, obtained by Juris Čerņenoks). Although it is oscillating almost as required by the Law of the Iterated Logarithm, very many of its values lay outside the segment.

Could we hope to prove the above estimates? To our knowledge, the best result on this problem is due to C. L. Stewart \cite{stewart}. It follows from his Theorem 2 (put $\alpha=0$), that
\[S_q(p^n)>\frac{\log{n}}{\log{\log{n}} +C_0}-1,\]
where the constant $C_0>0$ can be effectively computed from $q, p$. Since then, no better than $\frac{\log{n}}{\log\log{n}}$ lower bounds of $S_q(p^n)$ have been  proved. 

However, it appears that from a well-known unproved hypothesis about integer complexity in basis $\{1,+,\cdot\}$, one can derive a strong \textbf{linear} lower bound of $S_3(2^n)$.

\begin{proposition}
\label{ds_pno}
For any  primes $p, q$, and all $n$, $S_q(p^n) \geq \f{p^n} - n q \log_q{p}$.
\end{proposition}
\begin{proof}
Assume, $a_m a_{m-1}...a_0$ is a canonical base $q$ representation of the number $p^n$. One can derive from it a representation  of $p^n$ in basis $\{1,+,\cdot\}$, having length $\leq mq+S_q(p^n)$. Hence, $\f{p^n} \leq mq+S_q(p^n)$. Since $q^m \leq p^n < q^{m+1}$, we have $m \leq  n \log_q{p} < m+1$, and $\f{p^n} \leq n q \log_q{p}+S_q(p^n)$.
\end{proof}

\begin{theorem}
\label{log_pn}
If, for a prime $p \neq 3$, $\epsilon>0$, and $n>0$, $\fl{p^n} \geq 3+\epsilon$,
then $S_3(p^n) \geq n \epsilon \log_3{p}$.
\end{theorem}
\begin{proof}
Since
\[3+\epsilon \leq \fl{p^n}=\frac{\f{p^n}}{\log_3{p^n}},\]
according to Proposition \ref{ds_pno}, we have
\[S_3(p^n) \geq (3+\epsilon)n \log_3{p} - 3n \log_3{p} = n \epsilon \log_3{p}.\]
\end{proof}

Let us remind the well-known (and verified as true until $n=39$) \cite{paper1}
\begin{hypo}
\label{powtwohypo}
For all $n\geq 1$, $\f{2^n}=2n$ (moreover, the product of $1+1$'s is shorter than any other representation of $2^n$).
\end{hypo}

We consider proving or disproving of Hypothesis \ref{powtwohypo} as one of the biggest challenges of number theory.

If $\f{2^n}=2n$, then $\fl{2^n}=\frac{2}{\log_3{2}}$, and thus, by taking  in Theorem \ref{log_pn}, $\epsilon=\frac{2}{\log_3{2}}-3$, we obtain
\begin{cor}
\label{powtwocor}
If Hypothesis \ref{powtwohypo} is true, then for all $n>0$, $S_3(2^n) > 0.107\cdot n$. 
\end{cor}
Thus, proving of Hypothesis \ref{powtwohypo} would yield a strong linear lower bound for $S_3(2^n)$. Should this mean that proving of  Hypothesis \ref{powtwohypo} is an extremely complicated task?

Similar considerations appear in \cite{altman} (see the discussion following Conjecture 1.3) and \cite{altman3} (see Section 2.1.2).

\pdfbookmark[2]{2.2. Compression of powers}{compr}
\subsection{Compression of powers}
\label{subsec:compr}

For a prime $p$, can the shortest expressions of powers $p^n$ be obtained simply by multiplying the best expressions of $p$?

The answer ``yes'' can be proved easily for all powers of $p=3$. For example, the shortest expression of $3^3=27$ is $(1+1+1)\cdot(1+1+1)\cdot(1+1+1)$.
Thus, for all $n$, $\f{3^n}=n\f{3}=3n$.
The same seems to be true for the powers of  $p=2$, see the above Hypothesis \ref{powtwohypo}. For example, the shortest expression of $2^5=32$ is
$(1+1)\cdot(1+1)\cdot(1+1)\cdot(1+1)\cdot(1+1)$.
Thus, it seems, for all $n$, $\f{2^n}=n\f{2}=2n$.

However, for $p=5$ this is true only for $n=1, 2, 3 , 4, 5$, but the shortest expression of $5^6$  is not $5\cdot5\cdot5\cdot5\cdot5\cdot5$, but
\[5^6=15625=1+2^3 \cdot 3^2 \cdot 217=1+2^3 \cdot 3^2(1+2^3 \cdot 3^3).\]
Thus, we have here a kind of ``compression'': $\f{5^6}=29 < 6\f{5} = 30$.

Could we expect now that the shortest expression of $5^n$ can be obtained by multiplying the expressions of $5^1$ and $5^6$? This is true at least until $n=17$, as one can verify by using the online calculator \cite{calc} by Jānis Iraids. But, as observed by Juris Čerņenoks, $\f{5^{36}}$ is not $\f{5^6}\cdot 6 = 29\cdot 6 = 174$ as one might expect, but:
\[5^{36} = 2^4 \cdot 3^3 \cdot 247 \cdot 244125001 \cdot 558633785731 + 1,\]
where
\[247 = 3 \cdot (3^4 + 1) + 1;\]
\[244125001 = 2^3 \cdot 3^2 \cdot (2^3 \cdot 3^3 + 1) \cdot (2^3 \cdot 3^2 \cdot (2^3 \cdot 3^3 + 1) + 1) + 1;\]
\[558633785731 = 2 \cdot 3 \cdot (2^3 \cdot 3^5 + 1)\cdot(2 \cdot 3^4 \cdot (2^6 \cdot 3^5 \cdot (2 \cdot 3^2 + 1) + 1) + 1) + 1.\]
In total, this expression of $5^{36}$ contains $173$ ones.

Until now, no more ``compression points'' are known for powers of $5$.

Let us define the corresponding general notion:

\begin{definition}
Let us say that $n$ is a \textbf{compression point} for powers of the prime $p$, if and only if for any numbers $k_i$ such that $0<k_i<n$ and $\sum{k_i}=n$: \[\f{p^n}<\sum\f{{p^{k_i}}},\] 
i.e., if the shortest expression of $p^n$ is better than any product of expressions of smaller powers of $p$.
\end{definition}

\begin{que}
Which primes possess an infinite number of compression points, which ones - a finite number, and which ones do not possess them at all? 
\end{que}

Powers of 3 (and, it seems, powers of 2 as well) do not possess compression points at all. Powers of 5 possess at least two compression points. More about compression of powers of particular primes - see our previous paper \cite{paper1} (where compression is termed ``collapse'').

\begin{proposition}
If a prime $p\neq 3$ possess zero or finite number of compression points, then there is an $\epsilon>0$ such that for all $n>0$, $\fl{p^n} \geq 3+\epsilon$.
\end{proposition}
\begin{proof}
If $p\neq 3$, then for any particular $n$, $\fl{p^n} >3$.

If $n$ is not a compression point, then 
\[\f{p^n}=\sum\f{{p^{k_i}}}\]
for some numbers $k_i$ such that $0<k_i<n$ and $\sum{k_i}=n$. Now, if some of $k_i$-s is not a compression point as well, then we can express $\f{{p^{k_i}}}$ as $\sum\f{{p^{l_j}}}$, where $0<l_j<k_i$ and $\sum l_j=k_i$.

In this way, if m is the last compression point of $p$, then, for any $n>m$, we can obtain numbers $k_i$ such that $0<k_i\leq m$, $\sum{k_i}=n$, and 
\[\f{p^n}=\sum\f{{p^{k_i}}}.\]
Hence,
\[\fl{p^n}=\frac{\f{p^n}}{\log_3 p^n}=\frac{\sum\f{{p^{k_i}}}}{(\log_3 p)\sum k_i}.\]
Since, for any $a_i, b_i>0$,
\[\frac{\sum a_i}{\sum b_i}\geq \min \frac{a_i}{b_i},\]
we obtain that
\[\fl{p^n}\geq \min\frac{\f{p^{k_i}}}{k_i\log_3 p}=\min\fl{p^{k_i}}=3+\epsilon,\]
for some $\epsilon>0$.
\end{proof}

As we established in  Section \ref{subsec:conn}, for any particular prime $p \neq 3$, proving of $\fl{p^n} \geq 3+\epsilon$ for some $\epsilon>0$, and all sufficiently large $n>0$,  would yield a strong  linear lower bound for $S_3(p^n)$.  Therefore, for reasons explained in Section \ref{subsec:conn}, proving of the above inequality (even for a particular $p \neq 3$) seems to be an extremely complicated task. And hence, proving (even for a particular $p \neq 3$) that $p$ possess zero or finite number of compression points seems to be an extremely complicated task as well.

\begin{proposition}
\label{power_lim}
For any number $k$, $\lim \limits_{n \to \infty} \fl{k^n}$ exists, and does not exceed any particular $\fl{k^n}$.
\end{proposition}
\begin{proof}
Consider a number $n$ expressed as $n=mn_0+r$ where $m, n_0, r \in \mathbb{N}$.
\[\fl{k^n} \leq \frac{m \f{k^{n_0}}+\f{k^r}}{(mn_0+r)\log_3 k},\]
hence, for all $r$
\[\lim \sup \limits_{m \to \infty} \fl{k^{mn_0+r}} \leq \fl{k^{n_0}},\]
and consequently, for all $n_0$
\[\lim \sup \limits_{n \to \infty} \fl{k^n} \leq \fl{k^{n_0}}.\]
On the other hand, consider a subsequence of numbers $n_i$ such that
\[\lim \limits_{i \to \infty} \fl{k^{n_i}} = \lim \inf \limits_{n \to \infty} \fl{k^n}.\]
Since $\lim \sup \limits_{n \to \infty} \fl{k^n}$ does not exceed any of $\fl{k^{n_i}}$, we obtain that
\[\lim \sup \limits_{n \to \infty} \fl{k^n} = \lim \inf \limits_{n \to \infty} \fl{k^n}.\] 
\end{proof}
  
More about the spectrum of logarithmic complexity $\fl{n}$ see in our previous paper \cite{paper1}.

The \textbf{weakest possible hypothesis} about the spectrum of logarithmic complexities would be
\begin{hypo}
\label{weakhypo}
There is an $\epsilon>0$ such that for infinitely many numbers $n$: $\fl{n} \geq 3+\epsilon$.
\end{hypo}
Hypothesis \ref{weakhypo} should be easier to prove than Hypothesis \ref{powtwohypo} and other hypotheses from \cite{paper1}, but it remains still unproved nevertheless.

On the other hand,
\begin{que}
If, for all primes p, $\lim \limits_{n \to \infty} \fl{p^n}=3$, could this imply that, contrary to Hypothesis \ref{weakhypo}, $\lim \limits_{N \to \infty} \fl{N}=3$?
\end{que}

\pdfbookmark[1]{3. Integer complexity in basis $1,+, ., -$}{minus}
\section{Integer complexity in basis $\{1,+, \cdot, -\}$}
\label{minus}

In this Section, we consider representing of natural numbers by arithmetical expressions using 1's, addition,  multiplication, subtraction, and parentheses. According to Definition \ref{compl_minus}, $\fm{n}$ denotes the number of 1's in the shortest expressions representing $n$ in basis $\{1, +, \cdot, -\}$.

Of course, for all $n$, $\fm{n}\leq\f{n}$. The number $23$ is the first one, which possesses a better representation in basis $\{1, +, \cdot, -\}$ than in basis $\{1, +, \cdot\}$:   
\[23=2^3\cdot 3 -1 =2^2\cdot 5+2; \fm{23}=10; \f{23}=11.\]

\begin{definition}
\label{e}
\begin{enumerate}
\item[a)] Let's denote by $E_{-}(n)$ the  \textbf{largest} $m$ such that $\fm{m}=n$.
\item[b)] Let's denote by $E_{-k}(n)$ the \textbf{ $k$-th largest} $m$ such that $\fm{m}\leq n$ (if it exists). Thus, $E_{-}(n)=E_{-1}(n)$.
\item[c)] Let's denote by $e_{-}(n)$ the  \textbf{smallest} $m$ such that $\fm{m}=n$.
\end{enumerate}
\end{definition}

One can verify easily that $E_{-}(n)=E(n)$ for all $n>0$, i.e., that the formulas discovered by J. L. Selfridge for $E(n)$ remain valid for $E_{-}(n)$ as well:
\begin{proposition}
\label{e1minus}
For all $k\geq 0$:
\[E_{-}(3k+2) = 2\cdot 3^k;\]
\[ E_{-}(3k+3) = 3\cdot 3^k;\]
\[ E_{-}(3k+4) = 4\cdot 3^k.\]
\end{proposition}

One can verify also that for $n\geq 5, E_{-2}(n)=E_2(n)$, hence, the formula obtained by D. A. Rawsthorne \cite{raw} remains true for the basis $\{1, +, \cdot, -\}$: for all $n\geq 8$, $E_{-2}(n) = \frac{8}{9}E_{-}(n)$.

These formulas allow for building of feasible ``sieve'' algorithms for computing of $\fm{n}$. Indeed, after filtering out all $n$ with $\fm{n}<k$, one can filter out all $n$ with $\fm{n}=k$ knowing that $n\leq E_{-}(k)$, and trying out representations of $n$ as $A \cdot B, A+B, A-B$ for $A, B$ with $\fm{A}, \fm{B}<k$. See \cite{oeisA091333} for a more sophisticated efficient computer program designed by Jānis Iraids.

Juris Čerņenoks used another efficient program to compute $\fm{n}$ until $n=2 \cdot 10^{11}$. The program was written in Pascal, parallel processes were not used. With 64G RAM and additional 128G of virtual RAM (on SSD), the computation took 10 hours.

The values of $e_{-}(n)$ up to $n=81$ are represented in Table \ref{table:e1}. 

Some observations about $e_{-}(n)$ are represented in Table \ref{table:e} and Fig. \ref{e_n_minus2}. One might notice that the properties of the numbers around $e_{-}(n)$ are different from (and less striking than) the properties of the numbers around $e(n)$  \cite{paper1}.

Does Fig. \ref{e_n_minus2} provide some evidence that the logarithmic complexity of $n$ does not tend to $3$?

\begin{figure}
    \centering
    \includegraphics[width=0.8\textwidth]{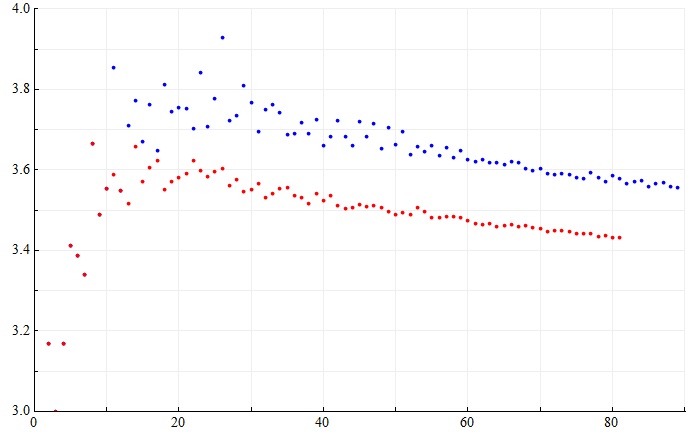}
    \caption{Logarithmic complexities of the numbers $e(n)$ (upper dots) and $e_{-}(n)$}
    \label{e_n_minus2}
\end{figure}

At least for all $2^n$ up to $2\cdot 10^{11}$ Hypothesis \ref{powtwohypo} remains true also for the basis $\{1, +, \cdot, -\}$.

While observing the shortest expressions representing small numbers in basis $\{1, +, \cdot, -\}$, one might conclude that whenever subtraction is the \textbf{last} operation of a shortest expression, then it is subtraction of $1$, for example, $23=2^3\cdot 3 -1$.

As established by Juris Čerņenoks, the first number, for which this observation \textbf{fails}, is larger than $55$ billions:

\[\fm{n}=75; n=55659409816 =(2^4 \cdot 3^3-1)(3^{17}-1)-2\cdot 3.\]

Until $2 \cdot 10^{11}$, there are only $3$ numbers, for which subtraction of 6 is necessary as the last operation of shortest expressions - the above one and the following two:
\[\fm{n}=77; n=111534056696 =(2^5\cdot 3^4-1)(3^{16}+1)-2\cdot 3,\]
\[\fm{n}=78; n=167494790108 =(2^4 \cdot 3^4+1)(3^{17}-1)-2 \cdot 3.\]

Necessity for subtraction of 8, 9, 12, or larger was not observed for numbers until  $2 \cdot 10^{11}$.
 
\begin{theorem}
\label{minus_main}
For all $n>1$, 
\[3\log_3 n \leq \fm{n} \leq 6\log_6 n + 5.890 < 3.679\log_3 n + 5.890,\]
If $n$ is a power of 3, then $\fm{n}=3\log_3 n$, else $\fm{n} > 3\log_3 n$.
\end{theorem}

\begin{proof}
The lower bound follows from Proposition \ref{e1minus}. Let us prove the upper bound.

If $n=6k$, then we can start building the expression for $n$ as $(1+1)(1+1+1)k$. Hence, by spending $5$ ones, we reduce the problem to building the expression for the number $k \leq \frac{n}{6}$. 

Similarly, if $n=6k+1$, then, by spending $6$ ones, we reduce the problem to building the expression for the number $k \leq \frac{n-1}{6}$. 

If $n=6k+2=2(3k+1)$, then, by spending $6$ ones, we reduce the problem to building the expression for the number $k \leq \frac{n-2}{6}$. 

If $n=6k+3=3(2k+1)$, then, by spending $6$ ones, we reduce the problem to building the expression for the number $k \leq \frac{n-3}{6}$. 

If $n=6k+4=2(3k+2)=2(3(k+1)-1)$, then, by spending $6$ ones, we reduce the problem to building the expression for the number $k+1 \leq \frac{n+2}{6}=\frac{n}{6}+\frac{1}{3}$. 

Finally, if $n=6k+5=6(k+1)-1$, then, by spending $6$ ones, we reduce the problem to building the expression for the number $k+1 \leq \frac{n+1}{6}=\frac{n}{6}+\frac{1}{6}$. 

Thus, by spending no more than $6$ ones, we can reduce building the expression for any number $n$ to building the expression for some number $k \leq \frac{n}{6}+\frac{1}{3}$. By applying this kind of operations $2$ times to the number $n$, we will arrive at a number $k \leq \frac{n}{6^2}+\frac{1}{6 \cdot 3}+\frac{1}{3}$. By applying them $m$ times, we will arrive at a number
\[k < \frac{n}{6^m}+\frac{1}{3}\cdot\frac{1}{1-\frac{1}{6}}=\frac{n}{6^m}+\frac{2}{5}.\]
Hence, if $\frac{n}{6^m}+\frac{2}{5} \leq 5$, or, $6^m \geq \frac{5n}{23}$, or $m \geq \log_6 \frac{5n}{23}$ , then, after $m$ operations, spending $\leq 6m$ ones, we will arrive at the number $\leq 5$. Thus,
\[\fm{n} \leq 6\left(\log_6 \frac{5n}{23} +1\right)+5 = 6\log_6 n +5.890< 3.679\log_3 n + 5.890.\]
\end{proof}

According to Theorem \ref{minus_main}, for all $n>1$:
\[3 \leq \fml{n} \leq 3.679+\frac{5.890}{\log_3 n}.\]

It seems, the largest values of $\fml{n}$ are taken by single numbers, see Table \ref{logc}. The lists in braces represent Cunningham chains of primes \cite{chains}.

\begin{center}
\begin{longtable}{|r|c|c|c|l|}
\caption{Largest values of $\fml{n}$}
\label{logc} \\
\hline
$n$ & $\fm{n}$ & $\approx \fml{n}$ & $\f{n}$ &Other properties \\
\hline
\endfirsthead
$11$ & $8$ & $3.665$ & $8$ & $e_{-}(8), \{2, 5, 11, 23, 47\}$ \\ 
$67$ & $14$ & $3.658$ & $14$ & $e_{-}(14)$, prime \\ 
$787$ & $22$ & $ 3.625$ & $22$ & $e_{-}(22)$, prime\\
$173$ & $17$ & $3.624$ & $17$ & $e_{-}(17)$, \{173, 347\} \\
$131$ & $16$ & $3.606$ & $16$ & $e_{-}(16), \{131, 263\}$ \\
$2767$ & $26$ & $3.604$ & $26$ & $e_{-}(26)$, prime\\
$2777$ & $26$ & $3.602$ & $26$ & $e_{-2}(26)$, prime \\
$823$ & $22$ & $3.600$ & $22$ & $e_{-2}(22)$, prime \\
$1123$ & $23$ & $3.598$ & $23$ & $e_{-}(23)$, prime \\
$2077$ & $25$ & $3.596$ & $25$ & $e_{-}(25), 31\cdot 67$ \\
$2083$ & $25$ & $3.594$ & $25$ & $e_{-2}(25)$, prime \\
$617$ & $21$ & $3.591$ & $21$ & $e_{-}(21)$, prime\\
$619$ & $21$ & $3.589$ & $21$ & $e_{-2}(21)$, prime\\
$29$ & $11$ & $3.589$ & $11$ & $e_{-}(11), \{29, 59\}$ \\
\hline
\end{longtable}
\end{center}

\begin{center}
\begin{longtable}{|c|r|c|r|c|r|c|r|}
\caption{$e_{-}(n)$}
\label{table:e1} \\
\hline
\multicolumn{1}{| c |}{$ n $} & \multicolumn{1}{ c |}{$e_{-}(n)$} & \multicolumn{1}{ c |}{$ n $} & \multicolumn{1}{ c |}{$e_{-}(n)$} & 
\multicolumn{1}{ c |}{$ n $} & \multicolumn{1}{ c |}{$e_{-}(n)$} & 
 \multicolumn{1}{ c |}{$ n $} & \multicolumn{1}{ c |}{$e_{-}(n)$}\\ \hline
\endfirsthead

1 & $1$ & 22 & $787$ & 43 & $718603$ & 64 & $666183787$\\ 
2 & $2$ & 23 & $1123$ & 44 & $973373$ & 65 & $913230103$\\   
3 & $3$ & 24 & $1571$ & 45 & $1291853$ & 66 & $1233996593$\\
4 & $4$ & 25 & $2077$ & 46 & $1800103$ & 67 & $1729098403$\\  
5 & $5$ & 26 & $2767$ & 47 & $2421403$ & 68 & $2334859277$\\   
6 & $7$ & 27 & $4153$ & 48 & $3377981$ & 69 & $3331952237$\\  
7 & $10$ & 28 & $5443$ & 49 & $4831963$ & 70 & $4649603213$\\  
8 & $11$ & 29 & $7963$ & 50 & $6834397$ & 71 & $6678905357$\\
9 & $17$ & 30 & $10733$ & 51 & $9157783$ & 72 & $9120679123$\\   
10 & $22$ & 31 & $13997$ & 52 & $12818347$ & 73 & $12457415693$\\  
11 & $29$ & 32 & $21101$ & 53 & $16345543$ & 74 & $17584630157$\\   
12 & $41$ & 33 & $27997$ & 54 & $23360983$ & 75 & $24864130483$\\ 
13 & $58$ & 34 & $36643$ & 55 & $34457573$ & 76 & $34145983337$\\  
14 & $67$ & 35 & $49747$ & 56 & $47377327$ & 77 & $47465340437$\\ 
15 & $101$ & 36 & $72103$ & 57 & $64071257$ & 78 & $68764257677$\\  
16 & $131$ & 37 & $99317$ & 58 & $87559337$ & 79 & $93131041603$\\  
17 & $173$ & 38 & $143239$ & 59 & $122103677$ & 80 & $132278645117$\\  
18 & $262$ & 39 & $179107$ & 60 & $174116563$ & 81 & $182226549067$\\ 
19 & $346$ & 40 & $260213$ & 61 & $247039907$ & & $$\\
20 & $461$ & 41 & $339323$ & 62 & $344781077$ & & $$\\
21 & $617$ & 42 & $508987$ & 63 & $467961763$ & & $$\\ 

\hline 
\end{longtable}
\end{center} 

\footnotesize
\begin{center}
\begin{longtable}{|c|r|r|r|l|}
\caption{Prime factorizations of numbers close to $e_{-}(n)$}
\label{table:e} \\
\hline
\multicolumn{1}{|c|}{$ n $} & \multicolumn{1}{c|}{$e_{-}(n)-2$} & \multicolumn{1}{c|}{$e_{-}(n)-1$} & \multicolumn{1}{c|}{$e_{-}(n)$} & \multicolumn{1}{c|}{$e_{-}(n)+1$} \\ \hline
\endfirsthead

1 & -- & -- & $1$ & $2$\\
2 & -- & $1$ & $2$ & $3$\\
3 & $1$ & $2$ & $3$ & $2^2$\\
4 & $2$ & $3$ & $2^2$ & $5$\\
5 & $3$ & $2^2$ & $5$ & $2\cdot3$\\
6 & $5$ & $2\cdot3$ & $7$ & $2^3$\\
7 & $2^3$ & $3^2$ & $2\cdot5$ & $11$\\
8 & $3^2$ & $2\cdot5$ & $11$ & $2^2\cdot3$\\
9 & $3\cdot5$ & $2^4$ & $17$ & $2\cdot3^2$\\
10 & $2^2\cdot5$ & $3\cdot7$ & $2\cdot11$ & $23$\\
11 & $3^3$ & $2^2\cdot7$ & $29$ & $2\cdot3\cdot5$\\
12 & $3\cdot13$ & $2^3\cdot5$ & $41$ & $2\cdot3\cdot7$\\
13 & $2^3\cdot7$ & $3\cdot19$ & $2\cdot29$ & $59$\\
14 & $5\cdot13$ & $2\cdot3\cdot11$ & $67$ & $2^2\cdot17$\\
15 & $3^2\cdot11$ & $2^2\cdot5^2$ & $101$ & $2\cdot3\cdot17$\\
16 & $3\cdot43$ & $2\cdot5\cdot13$ & $131$ & $2^2\cdot3\cdot11$\\
17 & $3^2\cdot19$ & $2^2\cdot43$ & $173$ & $2\cdot3\cdot29$\\
18 & $2^2\cdot5\cdot13$ & $3^2\cdot29$ & $2\cdot131$ & $263$\\
19 & $2^3\cdot43$ & $3\cdot5\cdot23$ & $2\cdot173$ & $347$\\
20 & $2^3\cdot17$ & $2^2\cdot5\cdot23$ & $461$ & $2\cdot3\cdot7\cdot11$\\
21 & $3\cdot5\cdot41$ & $2^3\cdot7\cdot11$ & $617$ & $2\cdot3\cdot103$\\
22 & $5\cdot157$ & $2\cdot3\cdot131$ & $787$ & $2^2\cdot197$\\
23 & $19\cdot59$ & $2\cdot3\cdot11\cdot17$ & $1123$ & $2^2\cdot281$\\
24 & $3\cdot523$ & $2\cdot5\cdot157$ & $1571$ & $2^2\cdot3\cdot131$\\
25 & $5^2\cdot83$ & $2^2\cdot3\cdot173$ & $31\cdot67$ & $2\cdot1039$\\
26 & $5\cdot7\cdot79$ & $2\cdot3\cdot461$ & $2767$ & $2^4\cdot173$\\
27 & $7\cdot593$ & $2^3\cdot3\cdot173$ & $4153$ & $2\cdot31\cdot 67$\\
28 & $5441$ & $2\cdot3\cdot907$ & $5443$ & $2^2\cdot1361$\\
29 & $19\cdot419$ & $2\cdot3\cdot1327$ & $7963$ & $2^2\cdot11\cdot181$\\
30 & $3\cdot7^2\cdot73$ & $2^2\cdot2683$ & $10733$ & $2\cdot3\cdot1789$\\
31 & $3^2\cdot5\cdot311$ & $2^2\cdot3499$ & $13997$ & $2\cdot3\cdot2333$\\
32 & $3\cdot13\cdot541$ & $2^2\cdot5^2\cdot211$ & $21101$ & $2\cdot3\cdot3517$\\
33 & $5\cdot11\cdot509$ & $2^2\cdot3\cdot2333$ & $27997$ & $2\cdot13999$\\
34 & $11\cdot3331$ & $2\cdot3\cdot31\cdot197$ & $36643$ & $2^2\cdot9161$\\
35 & $5\cdot9949$ & $2\cdot3\cdot8291$ & $49747$ & $2^2\cdot12437$\\
36 & $72101$ & $2\cdot3\cdot61\cdot197$ & $72103$ & $2^3\cdot9013$\\
37 & $3^2\cdot5\cdot2207$ & $2^2\cdot7\cdot3547$ & $99317$ & $2\cdot3\cdot16553$\\
38 & $227\cdot631$ & $2\cdot3\cdot23873$ & $143239$ & $2^3\cdot5\cdot3581$\\
39 & $5\cdot113\cdot317$ & $2\cdot3\cdot29851$ & $179107$ & $2^2\cdot44777$\\
40 & $3\cdot7\cdot12391$ & $2^2\cdot65053$ & $260213$ & $2\cdot3\cdot31\cdot1399$\\
41 & $3\cdot19\cdot5953$ & $2\cdot169661$ & $339323$ & $2^2\cdot3\cdot28277$\\
42 & $5\cdot101797$ & $2\cdot3^2\cdot28277$ & $508987$ & $2^2\cdot127247$\\
43 & $13\cdot167\cdot331$ & $2\cdot3\cdot229\cdot523$ & $718603$ & $2^2\cdot179651$\\
44 & $3\cdot7\cdot46351$ & $2^2\cdot243343$ & $973373$ & $2\cdot3\cdot162229$\\
45 & $3^2\cdot11\cdot13049$ & $2^2\cdot322963$ & $619\cdot2087$ & $2\cdot3\cdot215309$\\
46 & $1013\cdot1777$ & $2\cdot3\cdot300017$ & $1800103$ & $2^3\cdot83\cdot2711$\\
47 & $419\cdot5779$ & $2\cdot3\cdot403567$ & $2421403$ & $2^2\cdot131\cdot4621$\\
48 & $3^2\cdot11\cdot149\cdot229$ & $2^2\cdot5\cdot168899$ & $3377981$ & $2\cdot3\cdot562997$\\
49 & $17\cdot284233$ & $2\cdot3\cdot805327$ & $4831963$ & $2^2\cdot223\cdot5417$\\
50 & $5\cdot19\cdot71941$ & $2^2\cdot3\cdot569533$ & $6834397$ & $2\cdot3417199$\\
51 & $17\cdot199\cdot2707$ & $2\cdot3\cdot1526297$ & $9157783$ & $2^3\cdot1144723$\\
52 & $5\cdot31\cdot82699$ & $2\cdot3\cdot2136391$ & $12818347$ & $2^2\cdot29\cdot110503$\\
53 & $16345541$ & $2\cdot3\cdot2724257$ & $16345543$ & $2^3\cdot2043193$\\
54 & $7\cdot3337283$ & $2\cdot3\cdot3893497$ & $23360983$ & $2^3\cdot2920123$\\
55 & $3^2\cdot1259\cdot3041$ & $2^2\cdot17\cdot506729$ & $34457573$ & $2\cdot3\cdot5742929$\\
56 & $5^2\cdot1895093$ & $2\cdot3\cdot853\cdot9257$ & $79\cdot599713$ & $2^4\cdot2961083$\\
57 & $3\cdot5\cdot4271417$ & $2^3\cdot8008907$ & $64071257$ & $2\cdot3\cdot1193\cdot8951$\\
58 & $3^2\cdot5\cdot1945763$ & $2^3\cdot10944917$ & $87559337$ & $2\cdot3\cdot14593223$\\
59 & $3^2\cdot5^2\cdot542683$ & $2^2\cdot30525919$ & $122103677$ & $2\cdot3\cdot409\cdot49757$\\
60 & $37\cdot4705853$ & $2\cdot3\cdot29019427$ & $174116563$ & $2^2\cdot4349\cdot10009$\\
61 & $3\cdot5\cdot7\cdot2352761$ & $2\cdot123519953$ & $137\cdot1803211$ & $2^2\cdot3\cdot2683\cdot7673$\\
62 & $3\cdot5^2\cdot4597081$ & $2^2\cdot86195269$ & $344781077$ & $2\cdot3\cdot3823\cdot15031$\\
63 & $239\cdot1957999$ & $2\cdot3\cdot4931\cdot15817$ & $467961763$ & $2^2\cdot116990441$\\
64 & $5\cdot41\cdot811\cdot4007$ & $2\cdot3\cdot347\cdot319973$ & $666183787$ & $2^2\cdot166545947$\\
65 & $7^2\cdot18637349$ & $2\cdot3\cdot11059\cdot13763$ & $913230103$ & $2^3\cdot199\cdot573637$\\
66 & $3\cdot19\cdot223\cdot97081$ & $2^4\cdot77124787$ & $1233996593$ & $2\cdot3\cdot9337\cdot22027$\\
67 & $19\cdot91005179$ & $2\cdot3\cdot9431\cdot30557$ & $1729098403$ & $2^2\cdot11\cdot39297691$\\
68 & $3\cdot5^2\cdot7\cdot181\cdot24571$ & $2^2\cdot583714819$ & $2334859277$ & $2\cdot3\cdot389143213$\\
69 & $3^2\cdot5\cdot74043383$ & $2^2\cdot359\cdot2320301$ & $3331952237$ & $2\cdot3\cdot555325373$\\
70 & $3^2\cdot11\cdot46965689$ & $2^2\cdot1162400803$ & $4649603213$ & $2\cdot3\cdot774933869$\\
71 & $3^3\cdot5\cdot49473373$ & $2^2\cdot1669726339$ & $6678905357$ & $2\cdot3\cdot137\cdot8125189$\\
72 & $82301\cdot110821$ & $2\cdot3\cdot1520113187$ & $9120679123$ & $2^2\cdot2280169781$\\
73 & $3^2\cdot7\cdot1009\cdot195973$ & $2^2\cdot4327\cdot719749$ & $12457415693$ & $2\cdot3 \cdot2076235949$\\
74 & $3^2\cdot5\cdot281\cdot1390639$ & $2^2\cdot4396157539$ & $17584630157$ & $2\cdot3\cdot131\cdot22372303$\\
75 & $229\cdot1531\cdot70919$ & $2\cdot3\cdot4817\cdot860291$ & $24864130483$ & $2^2\cdot14779\cdot420599$\\
76 & $3\cdot5\cdot17\cdot5711\cdot23447$ & $2^3\cdot4268247917$ & $34145983337$ & $2\cdot3\cdot5690997223$\\
77 & $3^2\cdot5\cdot61\cdot17291563$ & $2^2\cdot1373\cdot8642633$ & $47465340437$ & $2\cdot3\cdot7910890073$\\
78 & $3^2\cdot5^2\cdot305618923$ & $2^2\cdot17191064419$ & $68764257677$ & $2\cdot3\cdot17\cdot674159389$\\
79 & $13\cdot193\cdot1033\cdot35933$ & $2\cdot3\cdot389\cdot39901903$ & $93131041603$ & $2^2\cdot23282760401$\\
80 & $3^2\cdot5\cdot3583\cdot820409$ & $2^2\cdot33069661279$ & $132278645117$ & $2\cdot3\cdot22046440853$\\
81 & $5\cdot 11 \cdot 1013 \cdot 3270691$ & $2 \cdot 3 \cdot 1613 \cdot 18828947$ & $182226549067$ & $2^2 \cdot 45556637267$\\ 

\hline 
\end{longtable}
\end{center} 
\normalsize

\pdfbookmark[1]{4. $P$-algorithms}{palgo}
\section{$P$-algorithms}
\label{palgo}

In this section we will explore a family of ``deterministic'' algorithms for building representations of numbers in basis $\{1,+,\cdot\}$. ``Deterministic'' means that these algorithms do not use searching over trees, but are building expressions directly from the numbers to be represented.

Let $P$ be a non-empty finite set of primes, for example, $P=\{2\}$, or $P=\{5, 11\}$. 

Let us define the following algorithm ($P$-algorithm). It is building an expression of a number $n>0$ in basis $\{1,+,\cdot\}$ by subtracting 1's and by dividing (whenever possible) by primes from the set $P$. More precisely, $P$-algorithm proceeds by applying of the following steps:

Step 1. If $n=1$ then represent $n$ as 1, and finish.

Step 2. If $n=p$ for some $p \in P$, then represent $n$ as $ex(p)$, where $ex(p)$ is some shortest expression of the number $p$ in basis $\{1,+,\cdot\}$, and finish.   
 
Step 3. If $n>1, n\notin P$ and $n$ is divisible by some $p \in P$, then represent $n$ as $ex(p)\cdot \frac{n}{p}$ (where $ex(p)$ is some shortest expression of the number $p$) and continue by processing the number $\frac{n}{p}$.

Step 4. If $n>1$ and $n$ is divisible by none of $p \in P$, then represent $n$ as $1+(n-1)$ and continue by processing the number $n-1$.

For example, consider the work of  the $\{5, 11\}$-algorithm:
\[157=1+1+1+11\cdot(1+1+1+1+5\cdot(1+1));\]
\[77=1+1+11\cdot(1+1+5).\]

\begin{definition}
The number of ones in the expression built by $P$-algorithm for the number $n$ does not depend on the order of application of Steps 1-4, let us denote this number by $\fp{n}$. The corresponding logarithmic complexity for $n>1$ is denoted by
$\fpl{n} = \frac{\fp{n}}{\log_3 n}$.
\end{definition}

For example, if $P=\{5, 11\}$:
\[\fp{157}=3+\f{11}+4+\f{5}+2=3+8+4+5+2=20;\]
\[\fp{77}=2+\f{11}+2+\f{5}=2+8+2+5=17.\]

Of course, for any $P$: $\fp{1}=1; \fp{2}=2;\fp{3}=3;\fp{4}=4;\fp{5}=5$.

\begin{proposition}
\label{p_algo_1}
(Lower bound) For any $P$, and all $n>1$,
 \[3 \leq \fl{n} \leq \f{n}_{P,log}.\]
This lower bound cannot be improved - the equality holds at least for $n=3$. 
\end{proposition}

\begin{hypo}
\label{p_algo_hypo}
(Upper bound) Let $q$ be the minimum number in $P$. Then, for all $n>1$,
\[\fpl{n} \leq \fl{q}+\frac{q-1}{\log_3 q}.\]
\end{hypo}

\begin{proposition}
\label{p_algo_partly}
The assertion of Hypothesis \ref{p_algo_hypo} holds, if the number q is such that for all $p \in P$:
\[\frac{\f{p}+q-2}{\log_q p} \leq \f{q}+q-1.\]
\end{proposition}
\begin{proof}
The assertion of the Hypothesis holds obviously for $n=2$. It holds also for $2<n \leq q-1$. Indeed, since $\frac{r}{\ln r}$ is growing at $r>e$, we have for these n,
 \[\frac{\fp{n}}{\log_3 n}=\frac{n}{\log_3 n} < \frac{q-1}{\log_3 q}.\]

So, let us assume that $n \geq q$ is the least number violating the inequality of the Hypothesis, namely:
\[\frac{\fp{n}}{\log_q n} > \f{q}+q-1.\]
Consider the last ``macro'' operation used by the $P$-algorithm to build the expression of the number $n$. It is either $r+pX$, where $0 \leq r \leq q-2; p \in P$, or $q-1+qX$. In either of cases a contradiction can be derived.
\end{proof}

Theorem \ref{p_algo_partly} allows to prove many cases of Hypothesis \ref{p_algo_hypo}.

1. $2 \in P$. Then $q=2$ and the condition of the Theorem holds obviously - it is well known that $\f{p} \leq 3 \log_2 p$ for all $p>1$.

2. $2 \notin P$ and $3\in P$. Then q=3, let us verify that $\frac{1+\f{p}}{\log_3 p} \leq \f{3}+3-1=5$ for all $p>3$. Since  $\f{p} \leq 3 \log_2 p$, we have: 
\[\frac{1+\f{p}}{\log_3 p} \leq \frac{1}{\log_3 p}+\frac{3}{\log_3 2}<\frac{1}{\log_3 p}+4.755,\]
hence, the required inequality holds for $p \geq 89$. As one can verify directly, it holds also for $3<p<89$ as well.

3. $q=5$. Let us verify that $\frac{3+\f{p}}{\log_5 p} \leq \f{5}+5-1=9$ for all $p>5$. Since  $\f{p} \leq 3 \log_2 p$, we have: 
\[\frac{3+\f{p}}{\log_5 p}=\frac{3}{\log_5 p}+\frac{\f{p}}{\log_3 p}\log_3 5 <\frac{3}{\log_5 p}+6.966,\]
hence, the required inequality holds for $p \geq 11$. As one can verify directly, it holds also for $3<p<11$ as well.

4. $q=7$. Let us verify that $\frac{5+\f{p}}{\log_7 p} \leq \f{7}+7-1=12$ for all $p>7$. Since  $\f{p} \leq 3 \log_2 p$, we have: 
\[\frac{5+\f{p}}{\log_7 p}=\frac{5}{\log_7 p}+\frac{\f{p}}{\log_3 p}\log_3 7 <\frac{5}{\log_7 p}+8.423,\]
hence, the required inequality holds for all $p \geq 16$. As one can verify directly, it holds also for $5<p<16$ as well.

5. $q=11$. Let us verify that $\frac{9+\f{p}}{\log_{11} p} \leq \f{11}+11-1=18$ for all $p>11$. Since  $\f{p} \leq 3 \log_2 p$, we have: 
\[\frac{9+\f{p}}{\log_{11} p}=\frac{9}{\log_{11} p}+\frac{\f{p}}{\log_3 p}\log_3 11 <\frac{9}{\log_{11} p}+10.379,\]
hence, the required inequality holds for all $p \geq 17$. As one can verify directly, it holds also for $7<p<17$ as well.

Unfortunately, this method does not generalize to all cases. The smallest prime number violating the condition of Theorem \ref{p_algo_partly}, is $q=163$. If we take $p=167$, then:
\[\frac{163-2+\f{167}}{\log_{163} 167}=\frac{161+17}{\log_{163} 167} > 177.156 > 163-1+\f{163}=162+15=177.\]

For the general case, we have proved a somewhat weaker

\begin{theorem}
\label{p_algo_weak}
Let $q$ be the minimum number in $P$, and $Q$ - the number in $P$ with the maximum $\fl{Q}$. Then, for all $n>1$, 
\[\fpl{n} \leq \fl{Q}+\frac{q-1}{\log_3 q}.\]
\end{theorem}
\begin{proof}
Consider the expression generated by the $P$-algorithm for the number $n$:
\[n=r_1+p_1 (r_2+p_2 (... (r_k+p_k \cdot r))),\]
where for all i: $p_i \in P; 0\leq r_i\leq q-1; 1\leq r \leq q-1$.
Then:
\[\fp{n}=\sum\limits_{i=1}^{k}r_i + \sum\limits_{i=1}^{k}\f{p_i}+r',\]
where $r'=0$, if $r=1$, else $r'=r$.

By setting all $r_i=0$ we obtain that $r \prod\limits_{i=1}^{k} p_i \leq n$, and that $r q^k \leq n$, or $k+ \log_q r \leq \log_q n$. 

Since
\[\frac{\f{p_i}}{\log_Q p_i} \leq \frac{\f{Q}}{\log_Q Q}=\f{Q},\]
we obtain,
\[\sum\limits_{i=1}^{k}\f{p_i} \leq \sum\limits_{i=1}^{k} \f{Q} \log_Q p_i = \f{Q} \log_Q \prod\limits_{i=1}^k p_i \leq \f{Q} \log_Q n\ = \frac{\f{Q}}{\log_3 Q} \log_3 n.\]

It remains to prove that the following expression does not exceed $q-1$:
\[\frac{\sum\limits_{i=1}^{k}r_i +r'}{\log_q n} \leq \frac{k(q-1)+r'}{k +\log_q r }.\]
If $r=1$, then $r'=0$, and the expression is equal to $q-1$, so, let us assume that $r'=r>1$ (then also $q \geq 3$), and let us apply the following general inequality that holds for any positive real numbers $a_j, b_j$:
\[\frac{\sum a_j}{\sum b_j} \leq \max \frac{a_j}{b_j}.\]
So, it remains to prove that $\frac{r}{\log_q r} \leq q-1$. This is obvious for $2<r \leq q-1$, since $\frac{r}{\ln r}$ is growing at $r>e$.

It remains to consider the situation $r=2$. Since $\frac{2}{\log_q 2} \leq q-1$ holds for $q\geq 7$, only two exceptions remain: $q=3$ and $q=5$. But these are covered by the above-mentioned consequences of Theorem \ref{p_algo_partly}.
\end{proof}

The spectrum of $\fpl{n}$ is characterized by the following

\begin{theorem}
\label{p_algo_main}
Let $q$ be the minimum number in $P$, and $p$ - the number in $P$ with the minimum $\fl{p}$. Then:

(1) The values of $\fpl{n}$ fill up densely the interval $\left(\fl{p}, \fl{q}+\frac{q-1}{\log_3 q}\right)$.

(2) For any $\epsilon>0$ there exist only finitely many $n$ such that $3 \leq \fpl{n}<\fl{p}-\epsilon$.
\end{theorem}

(1) and (2) of Theorem \ref{p_algo_main} follow from the lemmas below.

\begin{lemma}
\label{p_algo_main_2a}
Consider any two $p, q \in P, p < q$. Then the values of $\fpl{n}$ fill up densely the interval $(\fl{p}, \fl{q})$.
\end{lemma}

\begin{proof}
Consider, for any positive integers $a, b$, the logarithmic complexity of the number $p^a q^b$: 
\[\fpl{p^a q^b}=\frac{a\f{p}+b\f{q}}{a\log_3 p + b\log_3 q}.\]
Values of this expression fill up densely the interval
\[\left(\frac{\f{p}}{\log_3 p}, \frac{\f{q}}{\log_3 q}\right).\]
\end{proof}

\begin{lemma}
\label{p_algo_main_2b}
Let $q$ be the minimum number in $P$. Then, the values of $\frac{\fp{n}}{\log_q n}$ fill up densely the interval $(\f{q}, \f{q}+q-1)$. Hence, the values of $\fpl{n}$ fill up densely the interval $\left(\fl{q}, \fl{q}+\frac{q-1}{\log_3 q}\right)$.
\end{lemma}

\begin{proof}
We will build the necessary filling up numbers $n$ by using two operations on $X$: $qX$ and $q-1+qX$.

Let us start from a number $n_0$ such that  $n_0\equiv -1 \pmod p$ for all $p\in P$. By Chinese Remainder Theorem, there is such an $n_0 < \prod\limits_{p\in P} p$.

By Fermat's Little Theorem, if $p\in P$ and $p\neq q$, then $q^{p-1}\equiv 1 \mod p$. Hence, for $k=\prod\limits_{p\in P\backslash\{q\}} (p-1)$ and any $l$ we have $q^{kl}\equiv 1\pmod p$ for all $p\in P\backslash\{q\}$. Let us apply the operation $qX$  $kl$ times to the number $n_0$, thus obtaining the number $n_1=q^{kl}n_0\equiv -1 \mod p$ for all $p\in P\backslash\{q\}$. 

Let us note the following property of our second operation $q-1+qX$: for any $p\in P$, and any $X$: if $X\equiv -1 \pmod p$, then $q-1+qX\equiv -1 \pmod p$.

Hence, if we will build the number $n$ from the number $n_1$ by applying m times the operation $q-1+qX$, then 
\[n=(q-1)\sum\limits_{j=0}^{m-1}q^j+q^{m+kl}n_0=q^{m+kl}n_0+q^m-1,\]
and all the numbers X built in this process (with $n$ included) will possess the property $X\equiv p-1 \mod p$ for all $p\in P$.

And hence, when building an expression for the number $n$, $P$-algorithm will be forced, first, to apply $m$ times the operation $\frac{X-(q-1)}{q}$, spending for that $m(\f{q}+q-1)$ ones and reaching the number $n_1=q^{kl}n_0$.

After this, $P$-algorithm will be forced to apply $kl$ times the operation $\frac{X}{q}$, spending for that $kl\f{q}$ ones and reaching the number $n_0$, for which it will spend $\fp{n_0}$ ones.

Hence, $\fp{n}=m(\f{q}+q-1)+kl\f{q}+\fp{n_0}$.

On the other hand, 
\[\log_q n = m+kl+\log_q\left(n_0+\frac{q^m-1}{q^{kl+m}}\right)=m+kl+\log_q(n_0+q^{-kl}(1-q^{-m}));\]
\[\frac{\fp{n}}{\log_q n}=\frac{\f{q}+q-1+\frac{l}{m}k\f{q}+\frac{1}{m}\fp{n_0}}{1+\frac{l}{m}k+\frac{1}{m}\log_q(n_0+q^{-kl}(1-q^{-m}))}.\]
If, in this expression, $m$ and $\frac{l}{m}$ tend to infinity, then the expression tends to $\f{q}$. On the other hand, if $l=1$ and $m$ tends to infinity, then the expression tends to $\f{q}+q-1$.

But how about the intermediate points between $\f{q}$ and $\f{q}+q-1$? For any $\epsilon>0$, if $m$ is large enough, then
\[\left|\frac{\fp{n}}{\log_q n}-\frac{\f{q}+q-1+\frac{l}{m}k\f{q}}{1+\frac{l}{m}k}\right|<\epsilon.\]
As a function of a real variable $x$, the expression $h(x)=\frac{\f{q}+q-1+k\f{q}x}{1+kx}$, when x is growing from 0 to infinity, is decreasing continuously from $\f{q}+q-1$ to $\f{q}$. So, if we take $\frac{l}{m}$ close enough to $x$, then we will have $\left|\frac{\fp{n}}{\log_q n}-h(x)\right|<2\epsilon$.
\end{proof}

\begin{lemma}
\label{p_algo_main_2c}
Let $p$ be the number in $P$ with the minimum $\fl{p}$. Then for any $\epsilon>0$ there exist only finitely many $n$ such that $\fpl{n}<\fl{p}-\epsilon$.
\end{lemma}

\begin{proof}
Let us consider base $p$ logarithms instead of base 3. Assume the contrary: that for some $\epsilon>0$ there infinitely many numbers $n$ such that
\[\frac{\fp{n}}{\log_p n}<\frac{\f{p}}{\log_p p}-\epsilon=\f{p}-\epsilon,\] 
or, $\fp{n}<\f{p}\log_p n -\epsilon\log_p n$.

Following an idea proposed in \cite{altman}, let us define the ``$p$-defect'' of the number $n$ as follows: \[d_p(n)=\fp{n}-\f{p}\log_p n.\]
It follows from our assumption, that for infinitely many $n$,  $d_p(n)<-\epsilon\log_p n$, i.e., that $p$-defects can be arbitrary small (negative). Let us show that this is impossible.

Each positive integer can be generated by applying of two operations allowed by the $P$-algorithm. Let us consider, how these operations affect $p$-defects of the  numbers involved.

1. The operation $qX$, where $q\in P$. Then $\fp{qX}=\fp{X}+\f{q}$, and:
\[\begin{split}d_p(qX)&=\fp{qX}-\f{p}\log_p qX  \\
&=\f{q}+\fp{X}-\f{p}\log_p q - \f{p}\log_p X \\
&=d_p(X)+\f{q}-\f{p}\log_p q \\
&=d_p(X)+\f{q}\left(1-\frac{\f{p}\log_p q}{\f{q}\log_p p}\right)
\end{split}
.\]
Since $\frac{\f{p}}{\log_p p}\leq\frac{\f{q}}{\log_p q}$, we obtain that $d_p(qX)\geq d_p(X)$, i.e., that the operation $qX$ does not decrease the $p$-defect.

2. The operation $X+1$, where $X+1$  is not divisible by numbers of $P$. Then $\fp{X+1}=\fp{X}+1$, and:
\[\begin{split}d_p(X+1)&=\fp{X+1}-\f{p}\log_p(X+1)=\fp{X}+1-\f{p}\log_p(X+1)\\
&=d_p(X)+\f{p}\log_p(X)+1-\f{p}\log_p(X+1)\\
&=d_p(X)+1-\f{p}\log_p\frac{X+1}{X}
\end{split}
.\]
Hence, if $\f{p}\log_p\frac{X+1}{X}\leq 1$, then we obtain again that $d_p(X+1)\geq d_p(X)$. However, this will be true only, if $\log_p(1+\frac{1}{X})\leq \frac{1}{\f{p}}$, i.e., for all $X\geq \frac{1}{\sqrt[\f{p}]{p}-1}$ the operation $X+1$ does not decrease the $p$-defect.

The $p$-defect of the number $1$ is $d_p(1)=\f{1}-\f{p}\log_p 1 = 1$. Let us generate a tree, labeling its nodes with numbers. At  the root, let us start with the number 1, and, at each node, let us apply to the node's number all the possible operations $qX$ and $X+1$ allowed by $P$-algorithm, thus obtaining each time no more than $|P|+1$ new branches and nodes. Consider a particular branch in this tree: the numbers at its nodes are strongly increasing, but the corresponding $p$-defects may decrease. However, after $\frac{1}{\sqrt[\f{p}]{p}-1}$ levels $p$-defects will stop decreasing. So, in the entire tree, let us drop the nodes at levels greater than $\frac{1}{\sqrt[\f{p}]{p}-1}$. The remaining tree consists of a finite number of nodes, let us denote the minimum of the corresponding $p$-defects by $D$. Then, for all $n$, $d_p(n)\geq D$, which contradicts, for infinitely many $n$,  the inequality $d_p(n)<-\epsilon\log_p n$.
\end{proof}

\pdfbookmark[1]{5. Conclusion}{concl}
\section{Conclusion}
\label{concl}

Let us conclude with the summary of the most challenging \textbf{open problems}:

1) The \textbf{Question of Questions} - prove or disprove Hypothesis \ref{powtwohypo}: for all $n\geq 1$, $\f{2^n}=2n$, moreover, the product of $1+1$'s is shorter than any other representation of $2^n$, even in the basis with subtraction.

2) Basis $\{1,+,\cdot\}$. Prove or disprove the \textbf{weakest possible} Hypothesis \ref{weakhypo} about the spectrum of logarithmic complexity: there is an $\epsilon>0$ such that for infinitely many numbers $n$: $\fl{n} \geq 3+\epsilon$. An equivalent formulation: there is an $\epsilon>0$ such that for infinitely many numbers $n$: $\log_3 e(n) \leq (\frac{1}{3}-\epsilon)n$. Hypothesis \ref{powtwohypo} implies Hypothesis \ref{weakhypo}, so, the latter should be easier to prove?

3) Basis $\{1,+,\cdot, -\}$. Improve Theorem \ref{minus_main}: for all $n>1$,
 \[\fm{n} < 3.679\log_3 n + 5.890.\]

4) Solve the only remaining unsolved question about $P$-algorithms - prove or disprove Hypothesis \ref{p_algo_hypo}: let $q$ be the minimum number in $P$, then, for all $n>1$,
\[\fpl{n} \leq \fl{q}+\frac{q-1}{\log_3 q}.\]
It seems, an interesting number theory could arise here.

\bibliographystyle{splncs03}

\phantomsection
\addcontentsline{toc}{chapter}{References}

\end{document}